\documentclass[11pt]{amsart}
\usepackage{amsmath}
\usepackage[dvips]{epsfig}
\theoremstyle{plain}
\newtheorem{theorem}{Theorem}[section]
\newtheorem{lemma}[theorem]{Lemma}

\theoremstyle{definition}

\numberwithin{equation}{section}

\begin{document}

\title[Boundary Value Problems for Mixed Type Equations and Applications] {Boundary Value Problems for Mixed Type
Equations and Applications}

\author[Khuri]{Marcus A. Khuri}
\address{Department of Mathematics\\
Stony Brook University\\ Stony Brook, NY 11794}
\email{khuri@math.sunysb.edu}
\thanks{The author is partially supported by
NSF Grant DMS-1007156 and a Sloan Research Fellowship.}
\begin{abstract}
In this paper we outline a general method for finding well-posed boundary value problems
for linear equations of mixed elliptic and hyperbolic type, which extends previous techniques of Berezanskii, Didenko, and Friedrichs. This method is then used to
study a particular class of fully nonlinear mixed type equations which arise in applications to
differential geometry.
\end{abstract}
\maketitle

\section{Introduction}
\label{sec1}

An old classical problem from differential geometry asks, when can one realize a 2-dimensional Riemannian manifold, locally, in
3-dimensional Euclidean space? In other words, when can one ``see" an abstract surface, at least locally? As it turns
out, this question is equivalent to finding local solutions $z(x,y)$ to a Monge-Amp\`{e}re type equation, referred to
as the Darboux equation:
\begin{equation}\label{eqn4}
\det\nabla_{ij}z=K(\det h)(1-|\nabla_{h}z|^{2}).
\end{equation}
Here $h$ is the given Riemannian metric, $\nabla_{ij}$ are second covariant derivatives, and $K$ is the Gaussian
curvature of $h$. Another related problem is that of locally prescribing the Gaussian curvature of surfaces in 3-dimensional
Euclidean space. More precisely, given a function $K(x,y)$ defined in a neighborhood of the origin, does there exist
a graph $z=z(x,y)$ having Gaussian curvature $K$? Note that every surface may be expressed locally as a graph. This
problem is also equivalent to the local solvability of a Monge-Amp\`{e}re equation, namely
\begin{equation}\label{eqn5}
\det\partial_{ij}z=K(1+|\nabla z|^{2})^{2},
\end{equation}
where $\partial_{ij}$ are second partial derivatives.  In both equations (\ref{eqn4}) and (\ref{eqn5}), the sign of the
Gaussian curvature completely determines the type of the equation. When $K$ is positive the equation is elliptic, and
when $K$ is negative the equation is hyperbolic. Thus classical results may be used to analyze these problems in these
two situations. However when $K$ changes sign, the equation is of mixed type, and is very difficult to study. Nevertheless,
it can be shown \cite{HanKhuri} that by a suitable application of a Nash-Moser iteration, these two problems reduce to the
study of a linear equation having a particular form described below. More precisely,
in order to successfully apply the Nash-Moser iteration, one must find a well-posed boundary value problem for the associated linearized
equation, in a fixed domain about the origin, and establish certain a priori estimates. In previous work by Han, Hong, Lin, as well as the author,
this has been accomplished in the case for which the Gaussian curvature changes sign to finite order along a single smooth curve
(see \cite{Han1}, \cite{Han2}, \cite{HanHong}, \cite{Khuri1}, \cite{Lin}), and also in the case for which the Gaussian curvature vanishes to
finite order and has a zero set consisting of two transversely intersecting curves (see \cite{HanKhuri}, \cite{Khuri2}). Our
goal here is to extend these results by giving a general condition on the Gaussian curvature, which in particular allows for a change
of sign to infinite order and a zero set for $K$ which is more general than a finite number of intersecting curves. Counterexamples to the local
solvability of mixed type Monge-Amp\`{e}re equations, similar to, but not exactly of the form studied here, have been found \cite{Khuri3} in the
case of infinite order vanishing. Our main result is

\begin{theorem}\label{thm2}
Let $\varepsilon>0$ a small parameter. Suppose that the Gaussian curvature $K\in C^{\infty}$ satisfies the following condition
in a neighborhood of a point,
\begin{equation}\label{eqn7'}
\nabla_{V}K\geq\varepsilon(|\nabla K|+|K|)
\end{equation}
for some smooth vector field $V$. Then both equations (\ref{eqn4}) and (\ref{eqn5}) admit sufficiently smooth local solutions.
\end{theorem}

By a \textit{sufficiently smooth} local solution, we mean that for each sufficiently large integer $m$ there exists a neighborhood $\Omega_{m}$
such that the solution $z\in C^{m}(\Omega_{m})$. However, this does not necessarily imply that smooth local solutions exist,
since the size of the domains $\Omega_{m}$ may become arbitrarily small as $m\rightarrow\infty$. Note that condition (\ref{eqn7'})
will be satisfied for a wide variety of Gaussian curvatures. To see this, suppose that local coordinates $x$, $y$ have been
chosen near a point (corresponding to the origin in the $xy$-plane) such that the vector field $V$ is given by $\partial_{y}$.
Then we may take $K(x,y)=k(y)\phi(x,y)$ where $\phi>0$, $k(y)=\exp(-|y|^{-1})$ for $y>0$, $k(y)=-\exp(-|y|^{-1})$ for $y<0$, and $k(0)=0$.
In this example $K$ changes sign to infinite order across a single curve. However the zero set $K^{-1}(0)$ may be much more general.
For instance $K^{-1}(0)$ may be given by the region $|y|\leq |x|$; if $K>0$ for $y>|x|$ and $K<0$ for $y<-|x|$ then condition (\ref{eqn7'})
will be satisfied in a sufficiently small neighborhood of the origin.

Consider the following class of boundary value problems for linear second order
partial differential equations of the form:
\begin{align}\label{eqn1}
\begin{split}
Lu&=Ku_{xx}+u_{yy}+Au_{x}+Bu_{y}=f\text{ }\text{ }\text{ in }\text{ }\text{ }\Omega,\\
\mathcal{B}u&=\alpha u_{x}+\beta u_{y}+\gamma u=g\text{ }\text{ }\text{ on }\text{ }\text{ }\partial\Omega.
\end{split}
\end{align}
The coefficient functions of $L$ and $\mathcal{B}$ are assumed to be smooth in the domain $\Omega\subset\mathbb{R}^{2}$
and on its (piecewise smooth) boundary $\partial\Omega$, respectively. Moreover the function $K$ will be required to change
sign in $\Omega$, so that $L$ is of mixed elliptic and hyperbolic type. In the case that $K=y$ and $A=B=0$, $L$ is the well-known
Tricomi operator, which has been heavily studied in the context of transonic flows. The change from elliptic
to hyperbolic type as one crosses the $x$-axis represents the passing from subsonic to supersonic speeds. In \cite{Tricomi},
Tricomi studied the homogeneous equation ($f=0$) inside a domain bounded by a simple arc in the elliptic region $y>0$,
and two intersecting characteristic curves in the hyperbolic region $y<0$, which emanate from the two points where the
arc intersects $y=0$. Dirichlet boundary data, that is $\mathcal{B}u=u$, were then prescribed on the simple arc and
on one of the characteristic curves, leaving the other characteristic curve without any prescribed boundary conditions.
He was able to show that this boundary value problem is well-posed: it admits a unique regular solution, with continuous
dependence on the given data. Such problems may be described as \textit{open} boundary value problems, since the solution
is not prescribed in any way along some portion of the boundary. Open boundary value problems arise in flows in nozzles
and in other applications, and have received considerable attention. In contrast, \textit{closed}
boundary value problems, in which $\mathcal{B}u$ is prescribed on the whole boundary, are less well-studied. This lack
of attention is not due, however, to the absence of applications. For instance closed problems arise in constructing smooth flows
about airfoils. Rather, closed problems turn out to be more difficult to study, since they are often overdetermined for
regular solutions. In \cite{Payne}, Lupo, Morawetz, and Payne considered such closed problems for the Chaplygin equation,
where $A=B=0$ and $K=K(y)$ satisfies the condition
\begin{equation*}
K(0)=0,\text{ }\text{ }\text{ and }\text{ }\text{ }yK(y)>0,\text{ }\text{ }\text{ for }\text{ }\text{ }y\neq 0.
\end{equation*}
They showed the existence and uniqueness of weak solutions for the the Dirichlet and mixed Dirichlet-conormal
boundary value problems, with minimal restrictions on the boundary geometry of the domain. Previous results on
closed problems, often required restrictions on the boundary geometry or on the way in which $K$ changes sign
that were too strong to be of much help in applications to transonic fluid flows.

In this paper, we will study the case of homogeneous boundary conditions for problem (\ref{eqn1}), with very weak
restrictions on the possible ways in which $K$ changes sign. Unlike in the Tricomi case, when $K=y$, one cannot
ignore the lower order terms, and thus we will find appropriate conditions to impose on the functions $A$ and $B$
for which this problem is well-posed. Our goal is then to find a natural closed boundary value problem which admits
a unique, regular solution for each right-hand side $f$, and which admits appropriate a priori estimates to show
a strong continuous dependence on the given data. The domain will be taken to be a rectangle
\begin{equation}\label{eqn3}
\Omega=\{(x,y)\mid |x|<1,\text{ }\text{ }|y|<1\},
\end{equation}
however the two sides $x=\pm 1$ will be identified so that $\Omega$ becomes a cylinder. Thus all the functions
involved must be $2$-periodic. Altogether this has the effect of greatly simplifying the problem by eliminating
half of the boundary. On the remaining two portions of the boundary, conditions will be imposed as follows. On the
top of the cylinder $y=1$ (in the elliptic region), Dirichlet conditions $\mathcal{B}u=u$ will be fixed, while
on the bottom $y=-1$ (in the hyperbolic region), an oblique derivative condition $\mathcal{B}u=\alpha u_{x}+u_{y}$ will
be applied for some appropriately chosen constant $\alpha$ depending on $K$.

In the process of studying this problem, we will outline a general method for determining appropriate boundary
value problems for mixed type equations of the form (\ref{eqn1}). The procedure is in fact just a reorganized version
of the classical $a-b-c$ method of Friedrichs (also referred to as the multiplier method \cite{Payne1}) together with the techniques
of Berezanskii \cite{Berezanskii} and Didenko \cite{Didenko},
which involve global energy estimates and negative norm spaces.

In order to state our result for the linearized equation, let $\varepsilon>0$ be a small parameter, and let $\Omega$ be given by (\ref{eqn3}).
Consider the following boundary value problem
\begin{align}\label{eqn6}
\begin{split}
L_{\varepsilon}u&=\varepsilon Ku_{xx}+u_{yy}+\varepsilon Au_{x}+\varepsilon Bu_{y}=f\text{ }\text{ }\text{ in }\text{ }\text{ }\Omega,\\
& u(x,1)=0,\text{ }\text{ }\text{ }\text{ }(\alpha u_{x}+u_{y})(x,-1)=0,\text{ }\text{ }\text{ }\text{ }
u\text{ }\text{ is 2-periodic in $x$}.
\end{split}
\end{align}
We would like to point out that similar boundary conditions were studied by Han in \cite{Han1}, in the setting of a first order system
and where $K=y+O(\varepsilon)$.

The Sobolev space of square integrable derivatives up to and including order $m$, for functions
2-periodic in $x$, will be denoted by $H^{m}(\Omega)$, and its norm will be denoted by
$\parallel\cdot\parallel_{H^{m}(\Omega)}$. We will prove

\begin{theorem}\label{thm1}
Let $m$ be a nonnegative integer, $\varepsilon>0$ a small parameter, and $\alpha$ a constant. Suppose that the coefficients $K$, $A$, and $B$ are smooth,
2-periodic in $x$, and satisfy the following condition
\begin{equation}\label{eqn7}
K_{y}-\alpha K_{x}+2\alpha A \geq\varepsilon^{1/4}(|K_{x}|+|K|+|A|)\text{ }\text{ }\text{ in }\text{ }\text{ }\Omega.
\end{equation}
If $\alpha^{2}> -\varepsilon\min_{|x|\leq 1}K(x,-1)$, and $\varepsilon$ is sufficiently small, depending on $m$, $\alpha$ as well
as on the coefficients of $L$, then for each $f\in H^{m+1}(\Omega)$ there exists a unique solution $u\in H^{m}(\Omega)$ of
boundary value problem (\ref{eqn6}). Moreover, there exists a constant $C$ depending only on $m$ and the coefficients
of $L$ and their derivatives up to and including order $m$, such that
\begin{equation}\label{eqn8}
\parallel u\parallel_{H^{m}(\Omega)}\leq C\parallel f\parallel_{H^{m+1}(\Omega)}.
\end{equation}
\end{theorem}

We also remark that the solutions produced by Theorem \ref{thm1} actually possess slightly better regularity than is stated here.
This will become clear from the proof in Section \ref{sec3}.

This paper is organized as follows. In Section \ref{sec2} we review the required functional analysis, and introduce the
general procedure for ascertaining appropriate boundary conditions to obtain a well-posed problem. In Section \ref{sec3}
this procedure is used to treat (\ref{eqn6}), and to prove Theorem \ref{thm1}. Finally, the proof of our main result Theorem \ref{thm2}
is given in Section \ref{sec4}. An appendix, Section \ref{sec5}, contains proofs of some functional analysis results.

\section{Finding the Appropriate Boundary Conditions}\label{sec2}

We begin by introducing the necessary functional analysis needed to apply the general procedure
for ascertaining appropriate boundary conditions associated with a differential operator. Much of the
discussion in this section is expository, and is reorganized here for our particular application.

Frequently when dealing with mixed type equations, regularity will occur at different levels for different
directions, and it is then advantageous to have function spaces which can identify this
difference. Thus we will be working with the anisotropic Sobolev spaces $H^{(m,l)}(\Omega)$, which consist of
functions having square integrable derivatives up to and including order $m$ in the $x$-direction
and order $l$ in the $y$-direction. Here $\Omega$ is a domain in the $xy$-plane, and the norm on
these spaces is given by
\begin{equation*}
\parallel u\parallel_{(m,l)}^{2}=\int_{\Omega}\sum_{0\leq s\leq m\atop 0\leq t\leq l}
(\partial_{x}^{s}\partial_{y}^{t}u)^{2}.
\end{equation*}
We will also have need of the negative norm spaces of Lax \cite{Lax}. For each $v\in L^{2}(\Omega)$
the negative norms are given by
\begin{equation}\label{eqn}
\parallel v\parallel_{(-m,-l)}=\sup_{u\in H^{(m,l)}(\Omega)}\frac{|(u,v)|}{\parallel u\parallel_{(m,l)}},
\end{equation}
where $(\cdot,\cdot)$ denotes the $L^{2}(\Omega)$ inner product, and the spaces $H^{(-m,-l)}(\Omega)$ are defined
to be the completion of $L^{2}(\Omega)$ in this norm. Clearly
\begin{equation*}
\parallel v\parallel_{(-m,-l)}\leq \parallel v\parallel:=\parallel v\parallel_{(0,0)}\leq \parallel v\parallel_{(m,l)},
\end{equation*}
and so the following inclusions hold
\begin{equation*}
H^{(m,l)}(\Omega)\subset L^{2}(\Omega)\subset H^{(-m,-l)}(\Omega).
\end{equation*}
Moreover we have the generalized Schwarz inequality
\begin{equation}\label{eqn9}
|(u,v)|\leq\parallel u\parallel_{(m,l)}\parallel v\parallel_{(-m,-l)},\text{ }\text{ }\text{ }\text{ }
u\in H^{(m,l)}(\Omega),\text{ }\text{ }\text{ }\text{ }v\in H^{(-m,-l)}(\Omega).
\end{equation}
The negative norm spaces are important because they arise as the dual spaces to the Sobolev spaces.

Let $L$ be a linear partial differential operator, and consider the boundary value problem
\begin{equation}\label{eqn11}
Lu=f\text{ }\text{ }\text{ in }\text{ }\text{ }\Omega, \text{ }\text{ }\text{ }\text{ }\mathcal{B}u=0\text{ }\text{ }
\text{ on }\text{ }\text{ }\partial\Omega,
\end{equation}
and the associated adjoint problem
\begin{equation}\label{eqn12}
L^{*}v=g\text{ }\text{ }\text{ in }\text{ }\text{ }\Omega, \text{ }\text{ }\text{ }\text{ }\mathcal{B}^{*}v=0\text{ }\text{ }
\text{ on }\text{ }\text{ }\partial\Omega,
\end{equation}
where $\mathcal{B}$ is as in (\ref{eqn1}), $L^{*}$ is the formal adjoint of $L$, and the adjoint boundary conditions
$\mathcal{B}^{*}v=0$ are defined as follows. Let $C^{\infty}_{\mathcal{B}}(\overline{\Omega})$ denote the space of smooth functions (up to the
boundary) on $\Omega$ satisfying the boundary condition in (\ref{eqn11}). Then a function $v\in C^{1}(\overline{\Omega})$ is said to
satisfy the adjoint boundary conditions if $(Lu,v)=(u,L^{*}v)$ for all $u\in C^{\infty}_{\mathcal{B}}(\overline{\Omega})$. The space of smooth functions (up to the boundary) on $\Omega$ satisfying the boundary conditions of (\ref{eqn12}) will be denoted by $C^{\infty}_{\mathcal{B}^{*}}(\overline{\Omega})$.  Our first task is to find
an appropriate notion of weak solution for (\ref{eqn11}). We will say that $u\in H^{(m,l)}(\Omega)$ is a weak solution of (\ref{eqn11}), if
\begin{equation}\label{eqn13}
(u,L^{*}v)=(f,v)\text{ }\text{ }\text{ for all }\text{ }\text{ }v\in C^{\infty}_{\mathcal{B}^{*}}(\overline{\Omega}).
\end{equation}
Clearly a weak solution in $C^{2}(\overline{\Omega})$ satisfies (\ref{eqn11}) in the classical sense.

\begin{theorem}\label{thm3}
Let $m,l,s,t\in\mathbb{Z}_{\geq 0}$. There exists a weak solution $u\in H^{(m,l)}(\Omega)$ of (\ref{eqn11}) for each
$f\in H^{(s,t)}(\Omega)$, if and only if there exists a constant $C$ such that
\begin{equation}\label{eqn13}
\parallel v\parallel_{(-s,-t)}\leq C\parallel L^{*}v\parallel_{(-m,-l)}\text{ }\text{ }\textit{ for all }\text{ }\text{ }
v\in C^{\infty}_{\mathcal{B}^{*}}(\overline{\Omega}).
\end{equation}
\end{theorem}

This theorem generalizes a well-known result in the context of classical Sobolev spaces (see \cite{Berezanskii}) to the case of
the anisotropic Sobolev spaces. The proof requires only slight modification of the original and is thus relegated to the Appendix.
Moreover, this theorem shows that the problem of existence for (\ref{eqn11}) is reduced to establishing the
inequality (\ref{eqn13}). We now outline the basic procedure for accomplishing this goal.
This procedure will be implemented in the next section, for boundary value problem (\ref{eqn6}).

Let $v\in C^{\infty}_{\mathcal{B}^{*}}(\overline{\Omega})$, and consider an auxiliary boundary value problem
\begin{equation*}
Mu=v\text{ }\text{ }\text{ in }\text{ }\text{ }\Omega,\text{ }\text{ }\text{ }\text{ }\widetilde{\mathcal{B}}u=0
\text{ }\text{ }\text{ on }\text{ }\text{ }\partial\Omega,
\end{equation*}
where the differential operator $M$ and boundary operator $\widetilde{\mathcal{B}}$ are to be determined. The use of auxiliary
boundary value problems to study mixed type equations was first put forth by Didenko \cite{Didenko}. Note that
upon integrating by parts we have
\begin{equation}\label{eqn131}
(L^{*}v,u)-(v,Lu)=\int_{\partial\Omega}I_{1}(u,v),\text{ }\text{ }\text{ }\text{ }(Mu,Lu)=\int_{\Omega}I_{2}(u,u)
+\int_{\partial\Omega}I_{3}(u,u),
\end{equation}
for some quadratic forms $I_{1}$, $I_{2}$, and $I_{3}$. The goal is then to choose $M$, $\widetilde{\mathcal{B}}$, and
$\mathcal{B}^{*}$ appropriately so that
\begin{equation}\label{eqn14}
\int_{\Omega}I_{2}(u,u)\geq C^{-1}\parallel u\parallel^{2}_{(m,l)},
\end{equation}
\begin{equation}\label{eqn15}
\parallel v\parallel_{(-s,-t)}\leq C\parallel u\parallel_{(m,l)},
\end{equation}
and
\begin{equation}\label{eqn16}
\int_{\partial\Omega}(I_{1}(u,Mu)+I_{3}(u,u))\geq 0,
\end{equation}
where an additional integration by parts may be needed to obtain this last inequality. If this is successfully
achieved, then by applying the generalized Schwarz inequality, (\ref{eqn14}), and (\ref{eqn16}), we have
\begin{align*}
\parallel u\parallel_{(m,l)}\parallel L^{*}v\parallel_{(-m,-l)}&\geq(L^{*}v,u)\\
&=(v,Lu)+\int_{\partial\Omega}I_{1}(u,v)\\
&=(Mu,Lu)+\int_{\partial\Omega}I_{1}(u,v)\\
&=\int_{\Omega}I_{2}(u,u)+\int_{\partial\Omega}(I_{1}(u,Mu)+I_{3}(u,u))\\
&\geq C^{-1}\parallel u\parallel^{2}_{(m,l)}.
\end{align*}
The desired inequality (\ref{eqn13}) then follows from (\ref{eqn15}). In choosing the boundary conditions $\mathcal{B}^{*}$,
we note that the stronger the condition, the easier it is to establish (\ref{eqn16}), and hence existence. However a strong
condition $\mathcal{B}^{*}$ implies a weak condition $\mathcal{B}$, which could then make proving uniqueness for (\ref{eqn11})
difficult. Conversely, if the condition $\mathcal{B}^{*}$ is weak, then the condition $\mathcal{B}$ will be strong, which is an
advantageous situation for uniqueness but not existence. This just illustrates the intuitive fact, that a certain balance,
between existence and uniqueness, is needed when choosing boundary conditions in order to achieve a well-posed problem.

Lastly we point out how this procedure differs from the standard techniques. The first difference is the use of the anisotropic Sobolev spaces, while
the second difference concerns the use of inequality (\ref{eqn16}). Typically boundary conditions are chosen so that each of the boundary
integrals involving $I_{1}(u,v)$ and $I_{3}(u,u)$, vanish. This is of course much more restrictive than the requirement (\ref{eqn16}). It is
primarily this observation (that only (\ref{eqn16}) is needed) which allows us to establish the main theorems.

\section{Proof of Theorem \ref{thm1}}\label{sec3}

In this section we will study the following boundary value problem
\begin{align}\label{eqn17}
\begin{split}
L_{\varepsilon}u&=\varepsilon Ku_{xx}+u_{yy}+\varepsilon Au_{x}+\varepsilon Bu_{y}=f\text{ }\text{ }\text{ in }\text{ }\text{ }\Omega,\\
& u(x,1)=0,\text{ }\text{ }\text{ }\text{ }(\alpha u_{x}+u_{y})(x,-1)=0,\text{ }\text{ }\text{ }\text{ }
u\text{ }\text{ is 2-periodic in $x$},
\end{split}
\end{align}
where
\begin{equation*}
\Omega=\{(x,y)\mid |x|<1,\text{ }\text{ }|y|<1\},
\end{equation*}
and where $\alpha$ is a constant and all coefficients $K$, $A$, $B$, as well as the right-hand side $f$, are
2-periodic in $x$. The adjoint boundary value problem is given by
\begin{align}\label{eqn172}
\begin{split}
L_{\varepsilon}^{*}v&=\varepsilon Kv_{xx}+v_{yy}+\varepsilon(2K_{x}-A)v_{x}-\varepsilon Bv_{y}
+\varepsilon(K_{xx}-A_{x}-B_{y})v=g\text{ }\text{ }\text{ in }\text{ }\text{ }\Omega,\\
& v(x,1)=0,\text{ }\text{ }\text{ }\text{ }(\alpha v_{x}-v_{y})(x,-1)=0,\text{ }\text{ }\text{ }\text{ }
v\text{ }\text{ is 2-periodic in $x$}.
\end{split}
\end{align}
We will first establish existence for (\ref{eqn17}) in the appropriate spaces, under the assumption (\ref{eqn7}).
This will be accomplished by following the procedure from Section \ref{sec2}.

To begin, consider the auxiliary problem
\begin{align}\label{eqn173}
\begin{split}
Mu&=\sum_{s=0}^{m}(-1)^{s}\lambda^{-s}[\partial_{x}^{s}(a\partial_{x}^{s}u_{x})+\partial_{x}^{s}(b\partial_{x}^{s}u_{y})
+\partial_{x}^{s}(c\partial_{x}^{s}u)]=v\text{ }\text{ }
\text{ in }\text{ }\text{ }\Omega,\\
& u(x,1)=0,\text{ }\text{ }\text{ }\text{ }
u\text{ }\text{ is 2-periodic in $x$},
\end{split}
\end{align}
where $v\in C^{\infty}_{\mathcal{B}^{*}}(\overline{\Omega})$, and $a$, $b$, $c$ are functions to be given below (which are 2-periodic in $x$); in fact $b$ and $c$ will be
functions of $y$ alone. We claim that a unique smooth solution always exists. To see this, let
\begin{equation*}
w=\sum_{s=0}^{m}(-1)^{s}\lambda^{-s}\partial_{x}^{2s}u.
\end{equation*}
Clearly knowledge of $w$ yields knowledge of $u$. Thus we may create an iteration scheme in the following way, to find $u$.
Let $u_{0}=0$. Given $u_{i}$, solve
\begin{align*}
&a\partial_{x}w_{i+1}+b\partial_{y}w_{i+1}+cw_{i+1}=v-\sum_{s=0}^{m}(-1)^{s}\lambda^{-s}
\sum_{l=1}^{s}\left(\begin{array}{c}
s \\
l
\end{array}\right)\partial_{x}^{l}a\partial_{x}^{2s-l}(u_{i})_{x},\text{ }\text{ }
\text{ in }\text{ }\text{ }\Omega,\\
& w_{i+1}(x,1)=0,\text{ }\text{ }\text{ }\text{ }
w_{i+1}\text{ }\text{ is 2-periodic in $x$},
\end{align*}
for $w_{i+1}$ to obtain $u_{i+1}$. Note that this equation admits a unique smooth solution as long as $b\neq 0$ in
$\Omega$, according to the theory of first order partial differential equations. Moreover estimates are readily
available and can be used to show that the sequence $\{u_{i}\}$, so obtained, converges to the unique smooth solution
of (\ref{eqn173}).

Let $(n_{1},n_{2})$ denote the unit outer normal to $\partial\Omega$. In order to find the quadratic forms $I_{1}$, $I_{2}$,
and $I_{3}$ of (\ref{eqn131}), we integrate by parts and calculate
\begin{align}\label{eqn18}
\begin{split}
&(au_{x}+bu_{y}+cu,L_{\varepsilon}u)\\
&=\varepsilon\int_{\Omega}\frac{1}{2}[(bK)_{y}-2cK-(aK)_{x}+2aA]u_{x}^{2}
+[bA-(bK)_{x}-\varepsilon^{-1}a_{y}-aB]u_{x}u_{y}\\
&+\varepsilon\int_{\Omega}\frac{1}{2}[\varepsilon^{-1}(a_{x}-b_{y}-2c)+2bB]u_{y}^{2}
+\frac{1}{2}[(cK)_{xx}+\varepsilon^{-1}c_{yy}-(cA)_{x}-(cB)_{y}]u^{2}\\
& +\varepsilon\int_{\partial\Omega}\frac{1}{2}[aKn_{1}-bKn_{2}]u_{x}^{2}
+[bKn_{1}+\varepsilon^{-1}an_{2}]u_{x}u_{y}
+\varepsilon^{-1}\frac{1}{2}[bn_{2}-an_{1}]u_{y}^{2}\\
&+\varepsilon\int_{\partial\Omega}[cKn_{1}]uu_{x}+[\varepsilon^{-1}cn_{2}]uu_{y}
+\frac{1}{2}[cAn_{1}+cBn_{2}-(cK)_{x}n_{1}-\varepsilon^{-1}c_{y}n_{2}]u^{2},
\end{split}
\end{align}
\begin{align}\label{eqn19}
\begin{split}
&(Mu,L_{\varepsilon}u)\\
&=\sum_{s=0}^{m}\lambda^{-s}(a(\partial_{x}^{s}u)_{x}+b(\partial_{x}^{s}u)_{y}
+c(\partial_{x}^{s}u),L_{\varepsilon}(\partial_{x}^{s}u))\\
& +\sum_{s=0}^{m}\varepsilon\lambda^{-s}(a(\partial_{x}^{s}u)_{x}+b(\partial_{x}u)_{y}
+c(\partial_{x}^{s}u),\sum_{l=1}^{s}\left(\begin{array}{c}
s \\
l
\end{array}\right)(\partial_{x}^{l}K\partial_{x}^{s-l+2}u
+\partial_{x}^{l}A\partial_{x}^{s-l+1}u+\partial_{x}^{l}B\partial_{x}^{s-l}u_{y})),
\end{split}
\end{align}
and also
\begin{equation}\label{eqn20}
(L_{\varepsilon}^{*}v,u)-(v,L_{\varepsilon}u)=
\int_{\partial\Omega}[\varepsilon n_{1}Kv_{x}u
-\varepsilon n_{1}Kvu_{x}-n_{2}vu_{y}+n_{2}v_{y}u
+\varepsilon(n_{1}K_{x}-n_{1}A-n_{2}B)uv].
\end{equation}
Note that no boundary terms appear in (\ref{eqn19}) due to periodicity in the $x$-direction. According to the choice of the domain $\Omega$, we may disregard any boundary term with a factor of $n_{1}$.
Moreover, we will choose $b$ so that $b\neq 0$ in $\Omega$, and thus it is clear from (\ref{eqn172}) and (\ref{eqn173}) that
$u(x,1)=u_{y}(x,1)=0$. These two facts help simplify the expressions in (\ref{eqn18}), (\ref{eqn19}), and
(\ref{eqn20}). Furthermore by using the boundary condition $\mathcal{B}^{*}v=0$, and replacing $v$ with $Mu$,
in (\ref{eqn20}), we find that
\begin{align}\label{eqn21}
\begin{split}
&(L_{\varepsilon}^{*}v,u)-(v,L_{\varepsilon}u)\\
&=\int_{y=-1}(-n_{2}vu_{y}+n_{2}v_{y}u
-\varepsilon n_{2}Buv)\\
&=\int_{y=-1}v(u_{y}+\alpha u_{x}+\varepsilon Bu)\\
&=\int_{y=-1}\sum_{s=0}^{m}\lambda^{-s}(a\partial_{x}^{s}u_{x}
+b\partial_{x}^{s}u_{y}+c\partial_{x}^{s}u)
(\alpha\partial_{x}^{s}u_{x}+\partial_{x}^{s}u_{y}+\varepsilon B\partial_{x}^{s}u)\\
& +\int_{y=-1}\varepsilon\sum_{s=0}^{m}\lambda^{-s}(a\partial_{x}^{s}u_{x}
+b\partial_{x}^{s}u_{y}+c\partial_{x}^{s}u)\sum_{l=1}^{s}\left(\begin{array}{c}
s \\
l
\end{array}\right)\partial_{x}^{l}B\partial_{x}^{s-l}u\\
&=\int_{y=-1}\sum_{s=0}^{m}\lambda^{-s}[\alpha a(\partial_{x}^{s}u_{x})^{2}
+(a+\alpha b)(\partial_{x}^{s}u_{x})(\partial_{x}^{s}u_{y})- b(\partial_{x}^{s}u_{y})^{2}
-\frac{1}{2}\alpha c_{x}(\partial_{x}^{s}u)^{2}+c(\partial_{x}^{s}u)(\partial_{x}^{s}u_{y})]\\
& -\int_{y=-1}\varepsilon\sum_{s=0}^{m}\lambda^{-s}[\frac{1}{2}(aB)_{x}(\partial_{x}^{s}u)^{2}
-bB(\partial_{x}^{s}u)(\partial_{x}^{s}u_{y})-cB(\partial_{x}^{s}u)^{2}]\\
& +\int_{y=-1}\varepsilon\sum_{s=0}^{m}\lambda^{-s}(a\partial_{x}^{s}u_{x}
+b\partial_{x}^{s}u_{y}+c\partial_{x}^{s}u)\sum_{l=1}^{s}\left(\begin{array}{c}
s \\
l
\end{array}\right)\partial_{x}^{l}B\partial_{x}^{s-l}u.
\end{split}
\end{align}
Therefore by combining equations (\ref{eqn18}) (with various derivatives of $u$), (\ref{eqn19}), and (\ref{eqn21}) we obtain
\begin{align}\label{eqn22}
\begin{split}
&(L_{\varepsilon}^{*}v,u)\\
&=\sum_{s=0}^{m}\lambda^{-s}\varepsilon\int_{\Omega}\frac{1}{2}[(bK)_{y}-2cK-(aK)_{x}+2aA](\partial_{x}^{s}u_{x})^{2}
+[bA-(bK)_{x}-\varepsilon^{-1}a_{y}-aB](\partial_{x}^{s}u_{x})(\partial_{x}^{s}u_{y})\\
&+\sum_{s=0}^{m}\lambda^{-s}\varepsilon\int_{\Omega}\frac{1}{2}[\varepsilon^{-1}(a_{x}-b_{y}-2c)+2bB](\partial_{x}^{s}u_{y})^{2}
+\frac{1}{2}[(cK)_{xx}+\varepsilon^{-1}c_{yy}-(cA)_{x}-(cB)_{y}](\partial_{x}^{s}u)^{2}\\
& +\sum_{s=0}^{m}\varepsilon\lambda^{-s}(a(\partial_{x}^{s}u)_{x}+b(\partial_{x}^{s}u)_{y}
+c(\partial_{x}^{s}u),\sum_{l=1}^{s}\left(\begin{array}{c}
s \\
l
\end{array}\right)(\partial_{x}^{l}K\partial_{x}^{s-l+2}u
+\partial_{x}^{l}A\partial_{x}^{s-l+1}u+\partial_{x}^{l}B\partial_{x}^{s-l}u_{y}))\\
& +\sum_{s=0}^{m}\lambda^{-s}\int_{y=-1}[\alpha a+\frac{1}{2}\varepsilon bK](\partial_{x}^{s}u_{x})^{2}
+\frac{1}{2}b(\partial_{x}^{s}u_{y})^{2}+\alpha b(\partial_{x}^{s}u_{x})(\partial_{x}^{s}u_{y})\\
& +\sum_{s=0}^{m}\lambda^{-s}\int_{y=-1}\varepsilon bB(\partial_{x}^{s}u)(\partial_{x}^{s}u_{y})
+\frac{1}{2}[c_{y}+\varepsilon (cB-(aB)_{x})](\partial_{x}^{s}u)^{2}\\
& +\sum_{s=0}^{m}\lambda^{-s}\int_{y=-1}\varepsilon\sum_{l=1}^{s}\left(\begin{array}{c}
s \\
l
\end{array}\right)(a\partial_{x}^{s}u_{x}
+b\partial_{x}^{s}u_{y}+c\partial_{x}^{s}u)\partial_{x}^{l}B\partial_{x}^{s-l}u.
\end{split}
\end{align}

We are now ready to choose the functions $a$, $b$, and $c$. Let $\phi$ solve the following ODE
\begin{equation*}
\alpha\phi_{y}+\varepsilon\alpha B\phi=\varepsilon(A-K_{x})\text{ }\text{ }\text{ in }\text{ }\text{ }\Omega,\text{ }\text{ }
\text{ }\text{ }\phi(x,-1)=1.
\end{equation*}
Note that according to the definition of $\alpha$ in Theorem \ref{thm1}, $\phi=1+O(\varepsilon/|\alpha|)$.
Now set
\begin{equation}\label{eqn23}
a=\alpha\phi,\text{ }\text{ }\text{ }\text{ }b=1,\text{ }\text{ }\text{ }\text{ }c=-\varepsilon^{1/2}
+\varepsilon^{3/4}(3y+y^{2}).
\end{equation}
We immediately have
\begin{equation*}
(cK)_{xx}+\varepsilon^{-1}c_{yy}-(cA)_{x}-(cB)_{y}\geq \varepsilon^{-1/4}+O(\varepsilon^{1/2}),
\end{equation*}
\begin{equation*}
\varepsilon^{-1}(a_{x}-b_{y}-2c)+2bB\geq\varepsilon^{-1/2}+O(\varepsilon^{-1/4}),
\end{equation*}
and by the hypothesis (\ref{eqn7}) we also have
\begin{equation*}
(bK)_{y}-2cK-(aK)_{x}+2aA\geq 0.
\end{equation*}
By construction of $\phi$ it follows that the coefficient of the mixed derivative term $(\partial_{x}^{s}u_{x})(\partial_{x}^{s}u_{y})$,
in the second line of (\ref{eqn22}), is zero. There is however another mixed derivative term in the third line of this same
equation, however for $\varepsilon$ sufficiently small this is dominated by the sum of the two terms involving
$(\partial_{x}^{s}u_{x})^{2}$ and $(\partial_{x}^{s}u_{y})^{2}$. The remaining interior terms may be treated by one
more integration by parts, and by taking $\lambda$ sufficiently large. As for the boundary terms, we have that the
quadratic form
\begin{equation*}
[\alpha a+\frac{1}{2}\varepsilon bK](\partial_{x}^{s}u_{x})^{2}
+\alpha b(\partial_{x}^{s}u_{x})(\partial_{x}^{s}u_{y})+\frac{1}{2}b(\partial_{x}^{s}u_{y})^{2}
\end{equation*}
is positive, since
\begin{equation*}
\frac{1}{2}\left(\alpha^{2}+\frac{1}{2}\varepsilon K(x,-1)\right)-\frac{1}{4}\alpha^{2}>0
\end{equation*}
by the definition of $\alpha$. Moreover
\begin{equation*}
[c_{y}+\varepsilon (cB-(aB)_{x})](x,-1)=\varepsilon^{3/4}+O(\varepsilon),
\end{equation*}
and so the remaining boundary terms may be absorbed into those that are positive by taking $\varepsilon$ small,
and $\lambda$ large, after performing the appropriate integration by parts. Therefore there exists a constant $C>0$
such that
\begin{equation*}
(L^{*}_{\varepsilon}v,u)\geq C^{-1}\parallel u\parallel_{(m,1)}^{2}.
\end{equation*}
The generalized Schwarz inequality then yields
\begin{equation*}
\parallel L^{*}_{\varepsilon}v\parallel_{(-m,-1)}\geq C^{-1}\parallel u\parallel_{(m,1)}.
\end{equation*}
Furthermore, an integration by parts shows that
\begin{equation*}
\parallel v\parallel_{(-m-1,0)}\leq C\parallel u\parallel_{(m,1)},
\end{equation*}
and hence
\begin{equation*}
\parallel v\parallel_{(-m-1,0)}\leq C^{2}\parallel L^{*}_{\varepsilon}v\parallel_{(-m,-1)}.
\end{equation*}
Theorem \ref{thm3} may now be applied to boundary value problem (\ref{eqn17}), to obtain the existence of a
weak solution $u\in H^{(m,1)}(\Omega)$ for each $f\in H^{(m+1,0)}(\Omega)$.

We claim that this weak solution is in fact unique. This follows almost immediately from the calculations above.
Consider (\ref{eqn18}) with the same choices for $a$, $b$, and $c$ as in (\ref{eqn23}). The interior terms,
all together, are nonnegative, with the coefficient of $u^{2}$ positive, as we have shown. As for the boundary
integral, we may apply the boundary conditions $\mathcal{B}u=0$ (we are assuming here that $f\in C^{\infty}(\overline{\Omega})$
and hence $u\in C^{1}(\overline{\Omega})$, according to the additional regularity established below) and integrate by parts
to obtain
\begin{equation*}
\int_{y=1}\frac{1}{2}bu_{y}^{2}+\int_{y=-1}\frac{1}{2}[\varepsilon bK+2\alpha a-\alpha^{2}b]u_{x}^{2}
+\frac{1}{2}[c_{y}-\alpha c_{x}-\varepsilon cB]u^{2}.
\end{equation*}
This is clearly nonnegative. Therefore if $f=0$, we find that the only possible solution is $u=0$, and hence
uniqueness follows.

In order to obtain higher regularity for the solution given by Theorem \ref{thm3}, we will utilize the following
standard lemma concerning the difference quotient
\begin{equation*}
u^{q}(x,y):=\frac{u(x,y+q)-u(x,y)}{q}.
\end{equation*}

\begin{lemma}\label{lemma3}

$i)$ Let $u\in H^{(0,1)}(\Omega)$ and $\Omega'\subset\subset\Omega$ (that is, $\Omega'$ is compactly
contained in $\Omega$). Then
\begin{equation*}
\parallel u^{q}\parallel_{L^{2}(\Omega')}\leq\parallel
u_{y}\parallel_{L^{2}(\Omega)}
\end{equation*}
for all $0<|q|<\frac{1}{2}\mathrm{dist}(\Omega',\partial\Omega)$.
\smallskip

$ii)$  If $u\in L^{2}(\Omega)$ and $\parallel u^{q}\parallel_{L^{2}(\Omega')}\leq C$ for all
$0<|q|<\frac{1}{2}\mathrm{dist}(\Omega',\partial\Omega)$, then $u\in H^{(0,1)}(\Omega')$.

\end{lemma}

Let $u \in H^{(m,1)}(\Omega)$ be the weak solution given by Theorem \ref{thm3}, for $f\in H^{m+1}(\Omega)$.
We will show that in fact $u\in H^{m}(\Omega)$. If $m\leq 1$ then this statement follows trivially, so assume
that $m\geq 2$. We may integrate by parts to obtain
\begin{eqnarray*}
-(u_{y}+\varepsilon Bu,v_{y})&=&
(f-\varepsilon Ku_{xx}-\varepsilon Au_{x}+\varepsilon B_{y}u,v)\\
&
&+\int_{\partial\Omega}(\varepsilon Avu n_{1}-v_{y}u n_{2}-\varepsilon Kv_{x}un_{1}
-\varepsilon K_{x}vun_{1}+\varepsilon Kvu_{x}n_{1}),
\end{eqnarray*}
for all $v\in C^{\infty}_{\mathcal{B}^{*}}(\overline{\Omega})$. Note that since $u\in H^{1}(\Omega)$ we have that $u|_{\partial\Omega}$
is meaningful in $L^{2}(\partial\Omega)$, and in particular, as $\mathcal{B}u=0$, we have that $u(x,1)=0$ in the $L^{2}$-sense.
Moreover $u_{x}\in H^{1}(\Omega)$ and so $u_{x}|_{\partial\Omega}\in L^{2}(\partial\Omega)$. Thus we may integrate by parts
and use that $\mathcal{B}^{*}v=0$ in order to show that
\begin{equation*}
\int_{\partial\Omega}(\varepsilon Avu n_{1}-v_{y}u n_{2}-\varepsilon Kv_{x}un_{1}
-\varepsilon K_{x}vun_{1}+\varepsilon Kvu_{x}n_{1})=0.
\end{equation*}
We may then write
\begin{equation*}
(\overline{u},v_{y})=(\overline{f},v)\text{ }\text{
}\text{ for all }\text{ }\text{ }v\in C^{\infty}_{\mathcal{B}^{*}}(\overline{\Omega}),
\end{equation*}
where
\begin{equation*}
\overline{u}=-u_{y}-\varepsilon Bu,\text{
}\text{ }\text{ }\text{
}\overline{f}=f-\varepsilon Ku_{xx}-\varepsilon Au_{x}+\varepsilon B_{y}u.
\end{equation*}
Furthermore
\begin{equation*}
(\overline{u}^{q},v_{y})=(\overline{f}^{q},v)\text{
}\text{ }\text{ for all }\text{ }\text{ }v\in C_{c}^{\infty}(\Omega),
\end{equation*}
so that choosing a sequence $v_{i}\in C^{\infty}_{c}(\Omega)$ with
$v_{i}\rightarrow-\eta u^{q}$ in $H^{(0,1)}(\Omega)$ for
some nonnegative $\eta\in C^{\infty}_{c}(\Omega)$, implies that
\begin{eqnarray*}
\parallel\sqrt{\eta}\overline{u}^{q}\parallel^{2}&\leq&
|(\overline{f}^{q},\eta
u^{q})|+|(\overline{u}^{q},\eta_{y}u^{q})|
+|(\overline{u}^{q},\eta(\varepsilon Bu)^{q})|\\
&\leq&
\parallel\sqrt{\eta}\overline{f}^{q}\parallel\parallel\sqrt{\eta}u^{q}\parallel
+\parallel\sqrt{\eta}\overline{u}^{q}\parallel\parallel\frac{\eta_{y}}{\sqrt{\eta}}u^{q}\parallel
+\parallel\sqrt{\eta}\overline{u}^{q}\parallel\parallel\sqrt{\eta}(\varepsilon Bu)^{q}\parallel,
\end{eqnarray*}
where $\parallel\cdot\parallel$ denotes the $L^{2}(\Omega)$ norm. Then since $u,\overline{f}\in H^{(0,1)}(\Omega)$ and
$|\nabla\eta|^{2}\leq C\eta$, Lemma \ref{lemma3} $(i)$ yields $\parallel\sqrt{\eta}\overline{u}^{q}\parallel\leq C$ for
some constant $C$ independent of $q$, if $|q|$ is sufficiently small.  Now Lemma \ref{lemma3} $(ii)$ shows that $\overline{u}\in
H^{(0,1)}_{loc}(\Omega)$, as $\eta$ was arbitrary.  Hence $u_{yy}\in L^{2}_{loc}(\Omega)$.  It follows that the equation
$L_{\varepsilon}u=f$ holds in $L^{2}_{loc}(\Omega)$, and since we can solve for $u_{yy}$, we may boot-strap in the usual way
to obtain $u\in H^{m}(\Omega)$.

Lastly, to show that the solution $u$ satisfies the estimate (\ref{eqn8}), we recall the proof of uniqueness above. This proof
immediately gives
\begin{equation*}
(au_{x}+bu_{y}+cu,f)\geq C\parallel u\parallel^{2}_{(0,1)}.
\end{equation*}
Upon integrating by parts
\begin{equation*}
(au_{x},f)=-(u,af_{x}+a_{x}f),
\end{equation*}
and thus we have
\begin{equation*}
\parallel f\parallel_{(1,0)}\geq C\parallel u\parallel_{(0,1)}.
\end{equation*}
By differentiating equation (\ref{eqn17}) with respect to $x$, and applying a similar procedure, we find that
\begin{equation*}
\parallel f\parallel_{(m+1,0)}\geq C\parallel u\parallel_{(m,1)}.
\end{equation*}
By solving for $u_{yy}$ in equation (\ref{eqn17}), we may then estimate all remaining derivatives to obtain
the desired estimate (\ref{eqn8}). This completes the proof of Theorem \ref{thm1}.

\section{Proof of Theorem \ref{thm2}}\label{sec4}

Theorem \ref{thm2} follows almost immediately from Theorem \ref{thm1} and previous work. More precisely, as is shown in
\cite{HanKhuri}, the nonlinear problems (\ref{eqn4}) and (\ref{eqn5}) can be reduced to a study of the linearized equation via an application of the
Nash-Moser implicit function theorem. By an appropriate choice of coordinates (see \cite{HanKhuri}) the following may be arranged. First, the linearized
equation will have the form (\ref{eqn6}) where $A=K_{x}+\psi K$ for some smooth function $\psi$, and second, the vector field $V$ from (\ref{eqn7'}) will
be given by $\varepsilon^{7/8}(\partial_{y}+O(\varepsilon))$, where the parameter $\varepsilon$ represents a rescaling of the original coordinates and thus determines the size
of the domain of existence for the nonlinear equations. Moreover, since we are only concerned with local solutions for equations (\ref{eqn4}) and (\ref{eqn5}),
we may suitably modify the coefficients of the linearized equation away from the origin so that they are 2-periodic in $x$. Now also, (\ref{eqn7'}) implies that
(\ref{eqn7}) holds with $\alpha=O(\varepsilon^{1/2})$, for $\varepsilon$ sufficiently small. Therefore upon applying Theorem \ref{thm1} we obtain a unique solution satisfying
an a priori estimate. Lastly, in order to carry out the Nash-Moser iteration, a more precise a priori estimate, referred to as the Moser-estimate,
is needed. The Moser-estimate elucidates the dependence of the solution on the coefficients of the linearization, and is easily derived from the energy method of the previous section
(see \cite{HanKhuri}). This completes the proof of Theorem \ref{thm2}.

\section{Appendix}\label{sec5}

In this section we include a proof of Theorem \ref{thm3} for convenience of the reader. To begin recall that
the negative norm spaces arise as the dual spaces of Sobolev spaces.

\begin{lemma}\label{lemma1}
$H^{(-m,-l)}(\Omega)=H^{(m,l)}(\Omega)^{*}$.
\end{lemma}

\begin{proof}
For each $v\in L^{2}(\Omega)$ define a bounded linear function $F_{v}(u)=(u,v)$ on $H^{(m,l)}(\Omega)$.
We first show that the set
\begin{equation*}
\Lambda_{(m,l)}=\{F_{v}\in H^{(m,l)}(\Omega)^{*}\mid v\in L^{2}(\Omega)\}
\end{equation*}
is dense in $H^{(m,l)}(\Omega)^{*}$. To see this, observe that if $\Lambda_{(m,l)}$ is not dense, then there exists
$F\in H^{(m,l)}(\Omega)^{*}-\overline{\Lambda_{(m,l)}}$; here $\overline{\Lambda_{(m,l)}}$ denotes the closure of $\Lambda_{(m,l)}$.
According to a standard corollary of the Hahn-Banach Theorem, there then exists $\mathfrak{L}\in H^{(m,l)}(\Omega)^{**}$
such that $\mathfrak{L}(F)\neq 0$ and $\mathfrak{L}=0$ on $\overline{\Lambda_{(m,l)}}$. However by reflexivity
of Hilbert spaces there exists a nonzero $f\in H^{(m,l)}(\Omega)$ such that $\mathfrak{L}(\widetilde{F})
=\widetilde{F}(f)$ for all $\widetilde{F}\in H^{(m,l)}(\Omega)^{*}$. Thus $F_{v}(f)=0$ for all $F_{v}\in \Lambda_{(m,l)}$, which
implies that $(f,v)=0$ for all $v\in L^{2}(\Omega)$, so that $f=0$, a contradiction. This shows that $\Lambda_{(m,l)}$ is dense.

Now consider the map
\begin{equation*}
\mathfrak{I}:H^{(-m,-l)}(\Omega)\rightarrow H^{(m,l)}(\Omega)^{*}
\end{equation*}
defined in the following way. Each $v\in H^{(-m,-l)}(\Omega)$ arises as a limit $v=\lim_{n\rightarrow\infty}v_{n}$,
for some $v_{n}\in L^{2}(\Omega)$. We may then set $\mathfrak{I}(v)=\lim_{n\rightarrow\infty}F_{v_{n}}$, where
convergence is with respect to the operator norm. To see that this is well-defined, let $v=\lim_{n\rightarrow\infty}v_{n}
=\lim_{n\rightarrow\infty}\overline{v}_{n}$, and observe that since $\parallel v\parallel_{(-m,-l)}=\parallel F_{v}\parallel$ we have
\begin{equation*}
\parallel F_{v_{n}}-F_{\overline{v}_{n}}\parallel=\parallel F_{v_{n}-\overline{v}_{n}}\parallel=\parallel v_{n}-\overline{v}_{n}\parallel
\rightarrow 0.
\end{equation*}
To see that this map is one-to-one, suppose that $\mathfrak{I}(v)=\mathfrak{I}(w)$ then
\begin{equation*}
0=\lim_{n\rightarrow\infty}\parallel F_{v_{n}}-F_{w_{n}}\parallel
=\lim_{n\rightarrow\infty}\parallel v_{n}-w_{n}\parallel_{(-m,-l)}
=\parallel v-w\parallel_{(-m,-l)},
\end{equation*}
so that $v=w$. Also by the density property proved above, $\mathfrak{I}$ is onto. Lastly
\begin{equation*}
\parallel\mathfrak{I}(v)\parallel=\parallel F_{v}\parallel=\parallel v\parallel_{(-m,-l)}
\end{equation*}
so that $\mathfrak{I}$ is an isometric isomorphism.

\end{proof}

We may now construct an inner product on $H^{(-m,-l)}(\Omega)$. Let
\begin{equation*}
\mathfrak{F}:H^{(m,l)}(\Omega)^{*}\rightarrow H^{(m,l)}(\Omega)
\end{equation*}
be the isometric isomorphism given by the Riesz Representation Theorem. Then set
\begin{equation*}
(u,v)_{(-m,-l)}=(\mathfrak{F}\circ\mathfrak{I}(u),\mathfrak{F}\circ\mathfrak{I}(v))_{(m,l)},
\end{equation*}
where $(\cdot,\cdot)_{(m,l)}$ is the usual inner product on $H^{(m,l)}(\Omega)$. Note that
if $v_{n}\rightarrow v$ in $H^{(-m,-l)}(\Omega)$ then $F_{v_{n}}\rightarrow F_{v}$ with respect to the
operator norm, since for any $u\in H^{(m,l)}(\Omega)$,
\begin{equation*}
|(u,v-v_{n})|\leq \parallel u\parallel_{(m,l)}\parallel v-v_{n}\parallel_{(-m,-l)}\rightarrow 0.
\end{equation*}
This shows that every bounded linear functional on $H^{(m,l)}(\Omega)$ can be represented by $F_{v}$ for some
$v\in H^{(m,l)}(\Omega)$, and may be used to find that
\begin{equation*}
(v,v)_{(-m,-l)}=(\mathfrak{F}\circ\mathfrak{I}(v),\mathfrak{F}\circ\mathfrak{I}(v))_{(m,l)}
=\parallel\mathfrak{I}(v)\parallel=\parallel F_{v}\parallel
=\sup_{u\in H^{(m,l)}(\Omega)}\frac{|(u,v)|}{\parallel u\parallel_{(m,l)}}.
\end{equation*}
Therefore the inner product $(\cdot,\cdot)_{(-m,-l)}$ correctly generates the norm $\parallel\cdot\parallel_{(-m,-l)}$
given by (\ref{eqn}). We also note that since Hilbert spaces are reflexive, we could conclude from Lemma \ref{lemma1}
that $H^{(-m,-l)}(\Omega)^{*}=H^{(m,l)}(\Omega)^{**}=H^{(m,l)}(\Omega)$, however we would like a specific form of this
result.

\begin{lemma}\label{lemma2}
Any $G\in H^{(-m,-l)}(\Omega)^{*}$ may be represented by a unique $u\in H^{(m,l)}(\Omega)$, such that $G(v)=(u,v)$ for
all $v\in H^{(-m,-l)}(\Omega)$. In particular $H^{(-m,-l)}(\Omega)^{*}=H^{(m,l)}(\Omega)$.
\end{lemma}

\begin{proof}

Given $u\in H^{(m,l)}(\Omega)$ set $G_{u}(v)=(u,v)$, $v\in H^{(-m,-l)}(\Omega)$. By the generalized Schwarz inequality
(\ref{eqn9}), $\parallel G_{u}\parallel\leq \parallel u\parallel_{(m,l)}$ so that $G_{u}\in H^{(-m,-l)}(\Omega)^{*}$.
Moreover
\begin{equation*}
\parallel G_{u}\parallel=\sup_{v\in H^{(-m,-l)}(\Omega)}\frac{|(u,v)|}{\parallel v\parallel_{(-m,-l)}}
\geq \frac{|(u,v_{0})|}{\parallel v_{0}\parallel_{(-m,-l)}}=\frac{|F_{v_{0}}(u)|}{\parallel F_{v_{0}}\parallel},
\end{equation*}
where $v_{0}$ is chosen such that $F_{v_{0}}(u)=\parallel u\parallel_{(m,l)}$ and $\parallel F_{v_{0}}\parallel=1$.
This yields $\parallel G_{u}\parallel\geq\parallel u\parallel_{(m,l)}$, so we have $\parallel G_{u}\parallel=\parallel u\parallel_{(m,l)}$.

Consider the set
\begin{equation*}
\Lambda_{(-m,-l)}=\{G_{u}\in H^{(-m,-l)}(\Omega)^{*}\mid u\in H^{(m,l)}(\Omega)\}.
\end{equation*}
Then $\Lambda_{(-m,-l)}$ is dense in $H^{(-m,-l)}(\Omega)^{*}$. If not, then there exists $G\in H^{(-m,-l)}(\Omega)^{*}
-\overline{\Lambda_{(-m,-l)}}$. By a standard corollary of the Hahn-Banach Theorem there exists $\mathfrak{L}\in H^{(-m,-l)}(\Omega)^{**}$
such that $\mathfrak{L}(G)\neq 0$ and $\mathfrak{L}=0$ on $\overline{\Lambda_{(-m,-l)}}$. By reflexivity there is
a nonzero $f\in H^{(-m,-l)}(\Omega)$ with $\mathfrak{L}(\widetilde{G})=\widetilde{G}(f)$ for all $\widetilde{G}\in H^{(-m,-l)}(\Omega)^{**}$.
Thus $G_{u}(f)=0$ for all $G_{u}\in\Lambda_{(-m,-l)}$, which implies that $(u,f)=0$ for all $u\in H^{(m,l)}(\Omega)$, and hence
$f=0$, a contradiction.

Define a map
\begin{equation*}
\mathfrak{G}:H^{(m,l)}(\Omega)\rightarrow H^{(-m,-l)}(\Omega)^{*}
\end{equation*}
by $\mathfrak{G}(u)=G_{u}$. By the density property proved above, each $G\in H^{(-m,-l)}(\Omega)^{*}$ may be given by a limit
$G=\lim_{n\rightarrow\infty}G_{u_{n}}$, for some $u_{n}\in H^{(m,l)}(\Omega)$. Because $G_{u_{n}}$ converges and
$\parallel G_{u_{n}}\parallel=\parallel u_{n}\parallel_{(m,l)}$, we have that $u_{n}\rightarrow u$, and thus
$G(v)=(u,v)$ for all $v\in H^{(-m,-l)}(\Omega)$. That is, $\mathfrak{G}$ is onto. It is also clear that $\mathfrak{G}$ is
one-to-one, and $\parallel\mathfrak{G}(u)\parallel=\parallel G_{u}\parallel=\parallel u\parallel_{(m,l)}$, so that
$\mathfrak{G}$ is an isometric isomorphism.

\end{proof}

We now restate and give a proof of Theorem \ref{thm3}.

\begin{theorem}\label{thm4}
Let $m,l,s,t\in\mathbb{Z}_{\geq 0}$. There exists a weak solution $u\in H^{(m,l)}(\Omega)$ of (\ref{eqn11}) for each
$f\in H^{(s,t)}(\Omega)$, if and only if there exists a constant $C$ such that
\begin{equation}\label{eqn13'}
\parallel v\parallel_{(-s,-t)}\leq C\parallel L^{*}v\parallel_{(-m,-l)}\text{ }\text{ }\textit{ for all }\text{ }\text{ }
v\in C^{\infty}_{\mathcal{B}^{*}}(\overline{\Omega}).
\end{equation}
\end{theorem}

\begin{proof}
Suppose that the inequality (\ref{eqn13'}) holds, and consider the linear functional
\begin{equation*}
F:L^{*}C^{\infty}_{\mathcal{B}^{*}}(\overline{\Omega})=:X\rightarrow\mathbb{R}
\end{equation*}
given by
\begin{equation*}
F(L^{*}v)=(f,v),
\end{equation*}
for some fixed $f\in H^{(s,t)}(\Omega)$. Note that by the generalized Schwarz inequality and (\ref{eqn13'}),
\begin{equation*}
|F(L^{*}v)|\leq\parallel f\parallel_{(s,t)}\parallel v\parallel_{(-s,-t)}\leq C\parallel f\parallel_{(s,t)}
\parallel L^{*}v\parallel_{(-m,-l)},
\end{equation*}
and therefore $F$ is a bounded linear functional on the subspace $X\subset H^{(-m,-l)}(\Omega)$. The Hahn-Banach
Theorem then yields an extension $\widetilde{F}$ of $F$ to a bounded linear functional on all of $H^{(-m,-l)}(\Omega)$.
According to Lemma \ref{lemma2}, there then exists $u\in H^{(m,l)}(\Omega)$ such that
\begin{equation*}
\widetilde{F}(w)=(u,w)\text{ }\text{ }\text{ for all }\text{ }\text{ }w\in H^{(-m,-l)}(\Omega).
\end{equation*}
Upon restricting $w$ back to $X$, we obtain
\begin{equation*}
(u,L^{*}v)=\widetilde{F}(L^{*}v)=F(L^{*}v)=(f,v)\text{ }\text{ }\text{ for all }\text{ }\text{ }v\in
C^{\infty}_{\mathcal{B}^{*}}(\overline{\Omega}).
\end{equation*}

Conversely, assume that for any $f\in H^{(s,t)}(\Omega)$ there exists a weak solution $u_{f}\in H^{(m,l)}(\Omega)$, then
\begin{equation*}
|(f,v)|\leq|(u_{f},L^{*}v)|\leq\parallel u_{f}\parallel_{(m,l)}\parallel L^{*}v\parallel_{(-m,-l)}= C_{f}\parallel L^{*}v\parallel_{(-m,-l)}.
\end{equation*}
Consider the linear functional $G_{f}(v)=(f,v)$ on $H^{(-s,-t)}(\Omega)$. By the Riesz Representation Theorem we may write
$G_{f}(v)=(\mathfrak{F}^{-1}\circ\mathfrak{I}^{-1}(f),v)_{(-s,-t)}$ for some $\mathfrak{F}^{-1}\circ\mathfrak{I}^{-1}(f)\in H^{(-s,-t)}(\Omega)$.
From the proof of Lemmas \ref{lemma1} and \ref{lemma2} , we know that $\mathfrak{F}^{-1}\circ\mathfrak{I}^{-1}: H^{(s,t)}(\Omega)\rightarrow H^{(-s,-t)}(\Omega)$ is an isometry. Thus
\begin{equation*}
|(\mathfrak{F}^{-1}\circ\mathfrak{I}^{-1}(f),
v\parallel L^{*}v\parallel_{(-m,-l)}^{-1})_{(-s,-t)}|=|(f,v\parallel L^{*}v\parallel_{(-m,-l)}^{-1})| \leq C_{f}.
\end{equation*}
We now have a family of bounded linear functionals $J_{v}\in H^{(-s,-t)}(\Omega)^{*}$ given by
\begin{equation*}
J_{v}(u)=(u,v\parallel L^{*}v\parallel_{(-m,-l)}^{-1})_{(-s,-t)},
\end{equation*}
where $u=\mathfrak{F}^{-1}\circ\mathfrak{I}^{-1}(f)$ for some $f\in H^{(s,t)}(\Omega)$. This family is pointwise bounded for all
$v\in C^{\infty}_{\mathcal{B}^{*}}(\overline{\Omega})$, and therefore the Banach-Steinhaus Theorem asserts that this
family is uniformly bounded, that is, $\parallel J_{v}\parallel\leq C$ for all $v\in C^{\infty}_{\mathcal{B}^{*}}(\overline{\Omega})$.
However
\begin{equation*}
|J_{v}(u)|\leq\parallel\mathfrak{F}^{-1}\circ\mathfrak{I}^{-1}(f)\parallel_{(-s,-t)}\parallel v\parallel L^{*}v\parallel_{(-m,-l)}^{-1}\parallel_{(-s,-t)},
\end{equation*}
so that
\begin{equation*}
\parallel J_{v}\parallel\leq\parallel v\parallel L^{*}v\parallel_{(-m,-l)}^{-1}\parallel_{(-s,-t)}.
\end{equation*}
Also by choosing
\begin{equation*}
\mathfrak{F}^{-1}\circ\mathfrak{I}^{-1}(f)=v\parallel L^{*}v\parallel_{(-m,-l)}^{-1}
\end{equation*}
we obtain
\begin{equation*}
|J_{v}(v\parallel L^{*}v\parallel_{(-m,-l)}^{-1})|=\parallel v\parallel L^{*}v\parallel_{(-m,-l)}^{-1}\parallel_{(-s,-t)}^{2},
\end{equation*}
so that
\begin{equation*}
\parallel J_{v}\parallel\geq\parallel v\parallel L^{*}v\parallel_{(-m,-l)}^{-1}\parallel_{(-s,-t)}.
\end{equation*}
Hence
\begin{equation*}
\parallel J_{v}\parallel=\parallel v\parallel L^{*}v\parallel_{(-m,-l)}^{-1}\parallel_{(-s,-t)}.
\end{equation*}
Therefore the uniform bound yields
\begin{equation*}
\parallel v\parallel_{(-s,-t)}\leq C\parallel L^{*}v\parallel_{(-m,-l)}.
\end{equation*}

\end{proof}

\bibliographystyle{amsplain}

\end{document}